\documentclass[pdflatex,sn-mathphys-num]{sn-jnl}
\usepackage{graphicx}
\usepackage{multirow}
\usepackage{amsmath,amssymb,amsfonts}
\usepackage{amsthm}
\usepackage{mathrsfs}
\usepackage[title]{appendix}
\usepackage{xcolor}
\usepackage{textcomp}
\usepackage{manyfoot}
\usepackage{booktabs}
\usepackage{algorithm}
\usepackage{algorithmicx}
\usepackage{algpseudocode}
\usepackage{listings}
\theoremstyle{thmstyleone}
\newtheorem{theorem}{Theorem}
\newtheorem{corollary}{Corollary}

\theoremstyle{thmstyletwo}
\newtheorem{example}{Example}

\newtheorem{lemma}{Lemma}
\theoremstyle{thmstylethree}

\begin{document}
\title[On convergence of extended randomized Kaczmarz methods for matrix equations]{
On convergence of residual-based extended randomized Kaczmarz methods for matrix equations}
\author*[1]{\fnm{Wendi} \sur{Bao}}\email{baowd@upc.edu.cn}
\author[2]{\fnm{Jing} \sur{Li}}\email{17861271071@163.com}
\author[3]{\fnm{Lili} \sur{Xing}}\email{xinglily2010@upc.edu.cn}
\author[4]{\fnm{Weiguo} \sur{Li}}\email{liwg@upc.edu.cn}
\author[5]{\fnm{Jichao} \sur{Wang}}\email{wangjc@upc.edu.cn}
\affil{\orgdiv{College of Science}, \orgname{China University of Petroleum}, \orgaddress{\street{} \city{Qingdao}, \postcode{266580}, \state{} \country{P.R. China}}}

\abstract{In this paper, for solving inconsistent matrix equations we propose a dual-space residual-based randomized extended Kaczmarz method and its version with Nesterov momentum. Without the full column rank assumptions on coefficient matrices, we provide a thorough convergence analysis, and derive upper bounds for the convergence rates of the new methods. A feasible range for the momentum parameters is determined. Numerical experiments demonstrate that the proposed methods are much more effective than the existing ones, especially the method with momentum.
}



\maketitle

\section{Introduction}\label{sec1}

In this paper, we consider the following matrix equation
\begin{equation}\label{eq2}
	AX=B,  
\end{equation}
where $A\in R^{m\times n},B\in R^{m\times p}$. Here we will find its unique minimal $F$-norm least square solution $X^*=A^+B$, which implies that the system may be inconsistent. 
Such linear systems play a significant role in numerous mathematical and applied disciplines, such as algebra, differential equations, probability and statistics, biological sciences, and economics.
For solving matrix equations, many methods have been proposed (see some recent references, such as \cite{Numericalstrategies, Aniterative, Hermitian, Least, Matrix2010}). In \cite{Leastsquares}, for solving the operator least squares problem, Hajarian proposed the conjugate gradient least squares method via matrix-matrix products. Using the method in \cite{Leastsquares}, one can find the solution of the problem within a finite number of iterations in the absence of round-off errors. Many of these methods frequently use the matrix-matrix product operation, which consumes a large amount of computing time.

To reduce per-iteration computational cost, Wu et al.~\cite{Selection} developed two greedy Kaczmarz-type methods for the consistent matrix equation $AXB=C$. Despite the time-consuming greedy selection strategy, their core ideas are well-suited for large-scale consistent matrix equations. Wang and Song~\cite{Deterministic} later proposed an improved deterministic block Kaczmarz method for the same equation. Shafiei and Hajarian~\cite{Sylvester} extended the Kaczmarz method to matrix form for solving Sylvester equations via matrix-vector products. Li, Bao et al.~\cite{BaoLi} proposed a class of Kaczmarz-type methods for $AX=B$ and $XA=C$ without the assumption that the coefficient matrix is of full column rank. Here,
 the extended Kaczmarz methods (REKIAX) and the extended coordinate descent methods (REGSIAX) are obained for solving the inconsistent matrix equations $AX=B$. Recently, Ding and Huang~\cite{Quaternion} introduced a quaternion relaxed greedy randomized Kaczmarz method with adaptive parameters for quaternion matrix equations. 

It is well known that we can significantly improve the convergence rate of the algorithm by employing appropriate probability selection rules and momentum acceleration strategies. In 2025, under the assumption that the coefficient matrix is of full column rank, Chen and Gao~\cite{2025Accelerating} proposed the dual-space residual REK method for linear systems $Ax=b$,  achieving significantly accelerated convergence. Based on Nesterov's work, Sutskever et al. developed the applications of the Nesterov momentum method in deep learning \cite{2013momentum}. In~\cite{2024constrained}, Yin et al. combined Nesterov momentum with the Kaczmarz method and proposed a momentum-based constrained Kaczmarz method for image reconstruction problems.

Inspired by \cite{2024constrained,2025Accelerating,BaoLi}, we propose a dual-space residual-based random extended Kaczmarz (DREK, see Algorithm~\ref{algorithm6}) method and its Nesterov momentum accelerated (MDREK, see Algorithm~\ref{algorithm7}) method for solving the inconsistent matrix equation (\ref{eq2}). Without the assumption that the coefficient matrix is of full column rank, the convergence of the methods is investigated. Numerical experiments are presented to verify the efficiency of the new methods.

The notations in this paper are expressed as follows. For a matrix $A\in R^{m\times n}$, we use $A^T, ||A||_2, ||A||_F, A_{:,j_k}, A_{i_k,:}$ and $R(A)$ to represent the transpose, Euclidean norm, Frobenius norm, $j_k$-th column, $i_k$-th row and the range space of a matrix $A$, respectively.
$A_{:,j_k}^{(k)}, A_{i_k,:}^{(k)}$ represent the $j_k$-th column and $i_k$-th row at the $k$-th iteration selected. And denote $B^{\perp}$ is a matrix with the $j$th column such that $(B^{\perp})_{:,j}\in R(A)^{\perp}$.  For the right-hand side matrix $B$, there is the expression $B=\hat{B}+B^{\perp}$, where $\hat{B}_{:,j}\in R(A)$. In addition,
$\lambda_{min}(A^TA)$ and $\lambda_{max}(A^TA)$ represent its nonzero smallest and largest eigenvalues of the matrix $A^TA$, respectively. 

The structure of this paper is as follows. In Section~\ref{sec2}, we propose two new methods and provide their convergence analysis. In Section~\ref{sec3}, numerical experiments are presented to verify the effectiveness of the new methods. Finally, we present conclusions in Section~\ref{sec4}.

\section{Dual-space residual-based randomized extended Kaczmarz methods}\label{sec2}

To accelerate the convergence rate of methods for solving the linear system $Ax=b$, Chen and Gao~\cite{2025Accelerating} proposed the REK method with dual-space residuals (REKDR), which is described in Algorithm~\ref{algorithm2}.
Based on the REKDR method, the REKIAX in \cite{BaoLi}, and Nesterov momentum, we establish a dual-space residual-based random extended Kaczmarz method (DREK) and its momentum-accelerated variant (MDREK) for solving inconsistent matrix equation system \eqref{eq2}, detailed in Algorithms~\ref{algorithm6} and \ref{algorithm7}, respectively.

\begin{algorithm}
	\caption{REKDR method for $Ax=b$}\label{algorithm2}
	\begin{algorithmic}[1]
		\Statex \textbf{Input:} $A,b,l,x_0$  and $z_0=b$
		\Statex \textbf{Output:} $x_l$
		\State \textbf{For} {$k = 0$ to $l-1$} \textbf{do}
		\State Compute $\hat{r}_k=A^Tz_k$
		\State Select $j_k\in \{1,2,\cdots,n\}$ with probability $Pr(column=j_k)=\frac{|\hat{r}^{(j_k)}_k|^2}{||\hat{r}_k||_2^2}$
		\State Compute $z_{k+1}=z_k-\frac{A^T_{j_k}z_k}{||A_{(j_k)}||_2^2}A_{(j_k)}$
		\State Compute $r_k=b-Ax_k-z_k$
		\State Select $i_k\in {1,2,\cdots,m}$ with probability $Pr(row=i_k)=\frac{|r^{(i_k)}_k|^2}{||r_k||_2^2}$
		\State Compute $x_{k+1}=x_k+\frac{b^{(i_k)}-A^{(i_k)}x_k-z_{k+1}^{(i_k)}}{||A^{(i_k)}||_2^2}(A^{(i_k)})^T$
		\State \textbf{EndFor}
	\end{algorithmic}
\end{algorithm}

\begin{algorithm}
	\caption{ Dual-Space Residual-Based Randomized Extended Kaczmarz (DREK) method for $AX=B$ }\label{algorithm6}
	\begin{algorithmic}[1]
		\Statex \textbf{Input:} $A\in R^{m\times n},B\in R^{m\times p}, Y^{(0)}=X^{(0)}=0\in R^{n\times p},\ \gamma, l\in R$ and $Z^{(0)}=B$ 
		\Statex \textbf{Output:} $X^{(l)}$
		\State \textbf{For} {$k = 0$ to $l-1$} \textbf{do}
		\State  Set $Pr(column=j)=\frac{||\hat{r}_{:,j_k}||_2^2}{||\hat{r}||_F^2}$ with $\hat{r}_k=A^TZ^{(k)}$ 
		\State Compute $Z^{(k+1)}=Z^{(k)}-\frac{A_{:,j_k}}{||A_{:,j_k}||_2^2}((A_{:,j_k})^TZ^{(k)})$
		\State Set $Pr(row=i_k)=\frac{||r_{i_k,:}||_2^2}{||r||_F^2}$ with $r_k=B-AX^{(k)}-Z^{(k+1)}$
		\State Compute $X^{(k+1)}=X^{(k)}+\frac{(A_{i_k,:})^T(B_{i_k,:}-Z^{(k+1)}_{i_k,:}-A_{i_k,:}X^{(k)})}{||A_{i_k,:}||_2^2}$
		\State \textbf{endfor}
	\end{algorithmic}
\end{algorithm}

%%MREKDR2算法，加动量
\begin{algorithm}
	\caption{ Dual-Space Residual-Based Randomized Extended Kaczmarz method with Nesterov momentum (MDREK) for  $AX=B$}\label{algorithm7}
	\begin{algorithmic}[1]
		\Statex \textbf{Input:} $A\in R^{m\times n},B\in R^{m\times p}, Y^{(0)}=X^{(0)}\in R^{n\times p}, X^{(0)}_{:,j}\in R(A^T)(j=1,2,\cdots, p), Z^{(0)}=B$,\ and Parameter $\gamma$.
		\Statex \textbf{Output:} $X_l$
		\State \textbf{For} {$k =0$ to $l-1$} \textbf{do}
		\State  Set $Pr(column=j)=\frac{||\hat{r}_{:,j_k}||_2^2}{||\hat{r}||_F^2}$ with $\hat{r}_k=A^TZ^{(k)}$ 
		\State Compute $Z^{(k+1)}=Z^{(k)}-\frac{A_{:,j_k}}{||A_{:,j_k}||_2^2}((A_{:,j_k})^TZ^{(k)})$\label{for3}
		\State Set $Pr(row=i_k)=\frac{||r_{i_k,:}||_2^2}{||r||_F^2}$ with $r_k=B-AY^{(k)}-Z^{(k+1)}$
		\State Compute $X^{(k+1)}=Y^{(k)}+\frac{(A_{i_k,:})^T(B_{i_k,:}-Z^{(k+1)}_{i_k,:}-A_{i_k,:}Y^{(k)})}{||A_{i_k,:}||_2^2}$
		\State Compute $Y^{(k+1)}=X^{(k+1)}+\gamma(X^{(k+1)}-X^{(k)})$
		\State \textbf{endfor}
	\end{algorithmic}
\end{algorithm}

Next, we present several lemmas to establish the convergence of the new methods.

\begin{lemma}(\cite{2023OnBai})\label{lemma1}
    Let $m$ be a positive integer, $x_i\geq 0,y_i>0(i=1,2,\cdots,m)$. Then the following inequalities hold
    \begin{equation*}
        \sum_{i=1}^m\frac{x_i^2}{y_i}\geq \frac{(\sum_{i=1}^mx_i)^2}{\sum_{i=1}^my_i}, 
        \sum_{i=1}^mx_i^2\geq \frac{1}{m}(\sum_{i=1}^mx_i)^2.
    \end{equation*}
\end{lemma}

\begin{lemma}\label{lemma2}
    Let $\alpha_1,\beta_1$ be real numbers such as $\alpha_1\in (0,1)$ and $\beta_1-\alpha_1=\beta_1\alpha_1$.
    Then 
    \begin{equation*}
       ||X+Y||_F^2\geq \alpha_1||X||_F^2-\beta_1||Y||_F^2, \forall  X,Y\in R^{n\times p}.
    \end{equation*}
\end{lemma}
\begin{proof}
   Since $(ta - b)^2 \geq 0$, we have $ab \leq \frac{ta^2}{2} + \frac{b^2}{2t}$. 
Let $a = \|X\|_F$, $b = \|Y\|_F$, $t = 1 - \alpha_1 > 0$, where $\alpha_1 \in (0,1)$. 
Thus, it follows from the above mean inequality that
    \begin{equation*}
        ||X||_F||Y||_F\leq \frac{(1-\alpha_1)||X||_F^2}{2}+\frac{||Y||_F^2}{2(1-\alpha_1)}.
    \end{equation*}
    Combing with Cauchy-Schwarz inequality $|\langle X,Y\rangle_F|\leq ||X||_F||Y||_F$,we obtain
    \begin{align*}
       \langle X,Y\rangle_F
       &\geq -||X||_F||Y||_F\\
       &\geq -\frac{(1-\alpha_1)||X||_F^2}{2}-\frac{||Y||_F^2}{2(1-\alpha_1)}.
    \end{align*}
    Substituting the above inequality into the squared expansion yields
    \begin{align*}
        ||X+Y||_F^2
        &=||X||_F^2+2\langle X,Y\rangle_F+||Y||_F^2\\
        &\geq ||X||_F^2-(1-\alpha_1)||X||_F^2-\frac{||Y||_F^2}{1-\alpha_1}+||Y||_F^2\\
        &=\alpha_1||X||_F^2+(1-\frac{1}{1-\alpha_1})||Y||_F^2\\
        &\geq \alpha_1||X||_F^2-\frac{\alpha_1}{1-\alpha_1}||Y||_F^2\\
        &=\alpha_1||X||_F^2-\beta_1||Y||_F^2.
    \end{align*}
\end{proof}
 For the recurrence relation, refer to the proof of Lemma 9 in reference \cite{2017Stochastic}, we present a new generalized recurrence relation as follows, which plays an vital role in prove the convergence in Theorem \ref{thm2}.
\begin{lemma}\label{lemma3}(Generalized recursive relation)
    Let \( F_0 = F_1 \geq 0 \), and the non-negative sequence \( \{F_k\}_{k \geq 0} \) satisfy the following recursive relation:
    \begin{equation}\label{eq3}
        F_{k+1}\leq a_1F_k+a_2F_{k-1}+a_3,\forall k\geq1,
    \end{equation}
    where \( a_2 \geq 0 \), \( a_3 \geq 0 \), \( a_1 + a_2 < 1 \), and at least one of the coefficients \( a_1 \) and \( a_2 \) is positive.  If \( q = \frac{a_1 + \sqrt{a_1^2 + 4a_2}}{2} \) and \( \epsilon = q - a_1 \geq 0 \), then the following three conclusions hold:\\
    (1) $0<q<1$; \\
    (2) $q\geq a_1+a_2$, equality holds if and only if \( a_2 = 0 \);\\
    (3) $\forall k\geq1$, we can get
    \begin{equation*}
        F_k\leq q^{k-1}(F_1+\epsilon F_0)+\frac{a_3}{1-q}.
    \end{equation*}
   In particular, if \( F_1 = F_0 \), we have
    \begin{equation}\label{eq3}
        F_{k}\leq q^{k-1}(1+\epsilon)F_0+\frac{a_3}{1-q}.
    \end{equation}
\end{lemma}
\begin{proof}
First, we verify that \( \epsilon = \frac{-a_1 + \sqrt{a_1^2 + 4a_2}}{2} \) satisfies \( \epsilon \geq 0 \) and \( a_2 \leq (a_1 + \epsilon)\epsilon \). Since \( a_2 \geq 0 \) and the definition of $\epsilon$, it follows that \( \epsilon \geq 0 \). As \( \epsilon \) is a solution to the equation \( (a_1 + \epsilon)\epsilon - a_2 = 0 \), it also satisfies \( a_2 \leq (a_1 + \epsilon)\epsilon \).

Next, we verify that \( 0 < q < 1 \). Since \( q = a_1 + \epsilon = \frac{a_1 + \sqrt{a_1^2 + 4a_2}}{2} \), the non-negativity of \( q \) stems from the non-negativity of \( a_2 \). Therefore, as long as \( a_2 \geq 0 \), we have \( q \geq 0 \). Meanwhile, since at least one of the coefficients \( a_1 \) and \( a_2 \) is positive, \( q > 0 \) holds. Combining this with \( a_1 + a_2 < 1 \), we can derive that
\begin{equation*}
    q=\frac{a_1+\sqrt{a_1^2+4a_2}}{2}<\frac{1-a_2+\sqrt{(1-a_2)^2+4a_2}}{2}=1
\end{equation*}
Therefore, \( 0 < q < 1 \) holds.

Moreover, we verify that \( q \geq a_1 + a_2 \), with equality if and only if \( a_2 = 0 \). Since \( a_1 = q - \epsilon \), \( a_2 = q\epsilon \), and \( a_1 + a_2 < 1 \), we have \( a_1 + a_2 = q - \epsilon + q\epsilon = q + \epsilon(q - 1) \leq q \), which holds true.

    Finally, we verify Inequality (\ref{eq3}) of Lemma \ref{lemma3}. Given \( F_1 = F_0 \), adding \( \epsilon F_k \) to both sides yields
    \begin{align*}
        F_{k+1}
        &\leq F_{k+1}+\epsilon F_k\\
       &\leq a_1F_k+a_2F_{k-1}+a_3+\epsilon F_k\\
        &=q(F_k+\epsilon F_{k-1})+a_3\\
        &\leq q(q(F_{k-1}+\epsilon F_{k-2})+a_3)+a_3\\
        &=q^2(F_{k-1}+\epsilon F_{k-2})+a_3q+a_3\\
        &\leq \cdots\\
        &\leq q^{k}(F_1+\epsilon F_0)+a_3q^{k-1}+a_3q^{k-2}+\cdots+a_3q+a_3\\
        &\leq q^k(F_1+\epsilon F_0)+\frac{a_3}{1-q}\\
        &= q^k(1+\epsilon) F_0+\frac{a_3}{1-q}.
    \end{align*}
    Therefore, inequality (\ref{eq3}) of Lemma \ref{lemma3} also holds. 
\end{proof}
\begin{lemma}\label{thm1}
     Considering the coefficient matrix $A\in R^{m\times n}$ of the linear system
    $AX=B$, the vector $B\in R^{m\times p}$. Then, the MDREK method with the initial condition $Z^{(0)}=B$ holds that 
    \begin{equation}\label{eq8}
        E||Z^{(k)}-B^\perp||_F^2\leq \big(1-\frac{\lambda_{min}(A^TA)}{||A||_F^2}\big)^k\lambda_{max}(A^TA)||X_*||_F^2,k=1,2,\dots.
    \end{equation}
\end{lemma}
\begin{proof}
    Denote 
    \begin{equation*}
         P_{(j_k)}=I-\frac{A_{:,j_k}A_{:,j_k}^T}{||A_{:,j_k}||_2^2},j_k=1,2,\cdots,n.
    \end{equation*}
   Then we can get $P_{(j_k)}^2=P_{(j_k)},P_{(j_k)}^T=P_{(j_k)}.$
According to the third step of Algorithm \ref{algorithm7}, we obtain
\begin{equation*}
Z^{(k+1)}=Z^{(k)}-\frac{A_{:,j_k}(A_{:,j_k}^TZ^{(k)})}{||A_{:,j_k}||_2^2}=(I-\frac{A_{:,j_k}A_{:,j_k}^T}{||A_{:,j_k}||_2^2}) Z^{(k)}=P_{(j_k)}Z^{(k)}.
\end{equation*}
Due to $B^\perp\in R(A)^\perp$, so we can get $A^TB^\perp=0$. Thus we have
\begin{equation*}
    P_{(j_k)}B^\perp=B^\perp-\frac{A_{:,j_k}A_{:,j_k}^T}{||A_{:,j_k}||_2^2}B^\perp=B^\perp-0=B^\perp.
\end{equation*}
Furthermore, it is easy to get that 
\begin{equation}\label{eq5}
    Z^{(k+1)}-B^\perp=P_{(j_k)}Z^{(k)}-B^\perp=P_{(j_k)}Z^{(k)}-P_{(j_k)}B^\perp=P_{(j_k)}(Z^{(k)}-B^\perp).
\end{equation}
and
\begin{equation*}
    A_{:,j_k}^T(Z^{(k)}-B_\perp)=A_{:,j_k}^TZ^{(k)}=\hat{r}^{(k)}_{:,j_k}.
\end{equation*}
Taking the conditional expectation of the equation (\ref{eq5}), since $A^T_{:,j_k}B^{\perp}=0$, we can get 
\begin{align}\label{eq6}
    E_k||Z^{(k+1)}-B^\perp||_F^2 \notag
    &=E_k||P_{(j_k)}(Z^{(k)}-B^\perp)||_F^2\\ \notag
    &=E_k[trace\big((Z^{(k)}-B^\perp)^TP_{(j_k)}^TP_{(j_k)}(Z^{(k)}-B^\perp)\big)]\\ \notag
    &=E_k[trace\big((Z^{(k)}-B^\perp)^TP_{(j_k)}(Z^{(k)}-B^\perp)\big)]\\ \notag
&=\sum_{j_k=1}^n\frac{||A_{:,j_k}^TZ^{(k)}||_2^2}{||A^TZ^{(k)}||_F^2}trace\big((Z^{(k)}-B^\perp)^T(I-\frac{A_{:,j_k}A_{:,j_k}^T}{||A_{:,j_k}||_2^2})(Z^{(k)}-B^\perp)\big)\\ 
&=||Z^{(k)}-B^\perp||_F^2-\frac{1}{||A^TZ^{(k)}||_F^2}\sum_{j_k=1}^n\frac{||A_{:,j_k}^TZ^{(k)}||_2^4}{||A_{:,j_k}||_2^2}.
\end{align}
By Lemma \ref{lemma1}, we have
\begin{align} \label{eq7}\sum_{j_k=1}^n\frac{||A_{:,j_k}^TZ^{(k)}||_2^4}{||A_{:,j_k}||_2^2}\geq\frac{(\sum_{j_k=1}^n||A_{:,j_k}^TZ^{(k)}||_2^2)^2}{\sum_{j_k=1}^n||A_{:,j_k}||_2^2}=\frac{||A^TZ^{(k)}||_F^4}{||A||_F^2}.
\end{align}
Substituting (\ref{eq7}) into (\ref{eq6}), since $(Z^{(k)}-B^{\perp})_{:,j}\in R(A) (j=1,2,\cdots,p)$,  the following estimate  $||A^TZ^{(k)}-A^TB^{\perp}||_F^2\geq \lambda_{min}(A^TA)||Z^{(k)}-B^{\perp}||_F^2$ , we can obtain that
\begin{align*}
 E_k||Z^{(k+1)}-B^\perp||_F^2
&=||Z^{(k)}-B^\perp||_F^2-\frac{||A^TZ^{(k)}||_F^2}{||A||_F^2}\\
&=||Z^{(k)}-B^\perp||_F^2-\frac{||A^TZ^{(k)}-A^TB^\perp||_F^2}{||A||_F^2}\\
&\leq||Z^{(k)}-B^\perp||_F^2-\frac{\lambda_{min}(A^TA)||Z^{(k)}-B^\perp||_F^2}{||A||_F^2}\\
&=\big(1-\frac{\lambda_{min}(A^TA)}{||A||_F^2}\big)||Z^{(k)}-B^\perp||_F^2.
\end{align*}
Denote $w=1-\frac{\lambda_{min}(A^TA)}{||A||_F^2}$, because of $Z^{(0)}=B=\hat{B}+B^\perp$,  with full expectations, we have
\begin{align*}
    E||Z^{(k)}-B^\perp||_F^2
    &\leq wE||Z^{(k-1)}-B^\perp||_F^2\leq w^2E||Z^{(k-2)}-B^\perp||_F^2\leq\cdots\\
    &\leq w^{k-1}E||Z^{(1)}-B^\perp||_F^2\leq w^k||Z^{(0)}-B^\perp||_F^2=w^k||\hat{B}||_F^2.
\end{align*}
Since $||\hat{B}||_F^2\leq \lambda_{max}(A^TA)||X_*||_F^2$, then we obtain
\begin{equation*}
    E||Z^{(k)}-B^\perp||_F^2\leq w^k\lambda_{max}(A^TA)||X_*||_F^2.
\end{equation*}
\end{proof}
\begin{theorem}\label{thm2}
    Assume that the parameters $\alpha_1$ and $\beta_1$ are specified in Lemma \ref{lemma2}. Denote $\delta=1-\frac{\alpha_1^2\lambda_{min}(A^TA)}{||A||_F^2},\theta=\frac{\alpha_1\beta_1}{||A||_F^2}+\frac{1+\beta_1}{\lambda_{min}(A^TA)}, \eta_1=\delta(2\gamma^2+3\gamma+1)$, and $\eta_2=\delta\gamma(2\gamma+1)$. If $0\leq\gamma<\frac{1-\sqrt{\delta}}{2\sqrt{\delta}}$,  then for the initial matrices $Y^{(0)}=X^{(0)}$ with $X^{(0)}_{:,j}\in R(A^T)(j=1,2,\cdots, p)$ and $Z^{(0)}=B$, the iteration sequence $\{X^{(k)}\}_{k=0}^\infty$ generated by Algorithm \ref{algorithm7} converges to the least-squares solution $X_*=A^+B$. Subsequently, the solution error in expectation satisfies:
    \begin{equation*}
        E||X^{(k)}-X_*||_F^2\leq q^{(k)}(1+\xi)||X^{(0)}-X_*||_F^2+\frac{a_3}{1-q},
    \end{equation*}
    where 
$0<q=\frac{\eta_1+\sqrt{\eta_1^2+4\eta_2}}{2}<1$, $\xi=q-\eta_1\geq0$,  and $a_3=\theta \big(1-\frac{\lambda_{min}(A^TA)}{||A||_F^2}\big)^k\lambda_{max}(A^TA)||X_*||_F^2$.
\end{theorem}

\begin{proof}
Let $B=B^\perp+\hat{B}$, where $\hat{B}=AX_*\in R(A)$. From the iteration formula of $X^{(k)}$, we can get 
\begin{align*}
    X^{(k+1)}-X_*
    &=Y^{(k)}-X_*+\frac{A_{i_k,:}^T(B_{i_k,:}-A_{i_k,:}Y^{(k)}-Z_{i_k,:}^{(k+1)})}{||A_{i_k,:}||_2^2}\\
    &=Y^{(k)}-X_*+\frac{A_{i_k,:}^T(\hat{B}_{i_k,:}-A_{i_k,:}Y^{(k)}+B_{i_k,:}^\perp-Z_{i_k,:}^{(k+1)})}{||A_{i_k,:}||_2^2}\\
    &=Y^{(k)}-X_*-\frac{(A_{i_k,:})^TA_{i_k,:}(Y^{(k)}-X_*)}{||A_{i_k,:}||_2^2}+\frac{(A_{i_k,:})^T(B_{i_k,:}^\perp-Z_{i_k,:}^{(k+1)})}{||A_{i_k,:}||_2^2}\\
    &=\big(I-\frac{(A_{i_k,:})^TA_{i_k,:}}{||A_{i_k,:}||_2^2}\big)(Y^{(k)}-X_*)+\frac{(A_{i_k,:})^T(B_{i_k,:}^\perp-Z_{i_k,:}^{(k+1)})}{||A_{i_k,:}||_2^2}.\\
\end{align*}
By calculation, we know that 
\begin{equation*}
    A_{i_k,:}\big(I-\frac{(A_{i_k,:})^TA_{i_k,:}}{||A_{i_k,:}||_2^2}\big)(Y^{(k)}-X_*)=0.
\end{equation*}
Thus, we can get $\big(I-\frac{(A_{i_k,:})^TA_{i_k,:}}{||A_{i_k,:}||_2^2}\big)(Y^{(k)}-X_*)$ is orthogonal to
$\frac{(A_{i_k,:})^T(B_{i_k,:}^\perp-Z_{i_k,:}^{(k+1)})}{||A_{i_k,:}||_2^2}$.
Taking the Frobenius norm on both sides yields
\begin{align*}
    ||X^{(k+1)}-X_*||_F^2
    &=||\big(I-\frac{(A_{i_k,:})^TA_{i_k,:}}{||A_{i_k,:}||_2^2}\big)(Y^{(k)}-X_*)||_F^2+||\frac{(A_{i_k,:})^T(B_{i_k,:}^\perp-Z_{i_k,:}^{(k+1)})}{||A_{i_k,:}||_2^2}||_F^2\\
    &=||\big(I-\frac{(A_{i_k,:})^TA_{i_k,:}}{||A_{i_k,:}||_2^2}\big)(Y^{(k)}-X_*)||_F^2+\frac{||(A_{i_k,:})^T(B_{i_k,:}^\perp-Z_{i_k,:}^{(k+1)})||_F^2}{||A_{i_k,:}||_2^4}\\
    &\leq ||\big(I-\frac{(A_{i_k,:})^TA_{i_k,:}}{||A_{i_k,:}||_2^2}\big)(Y^{(k)}-X_*)||_F^2+\frac{||B_{i_k,:}^\perp-Z_{i_k,:}^{(k+1)}||_F^2}{||A_{i_k,:}||_2^2}\\
\end{align*}
Here, for the last inequality, since $||(A_{i_k,:})^T(B_{i_k,:}^\perp-Z_{i_k,:}^{(k+1)})||_F^2\leq ||A_{i_k,:}||_2^2||B_{i_k,:}^\perp-Z_{i_k,:}^{(k+1)}||_2^2$ is hold.

Taking the conditional expectation of the $k$-th iterations on both sides of the above equation, we can get
\begin{align*}
    E_k||X^{(k+1)}-X_*||_F^2
    &=E_k||\big(I-\frac{(A_{i_k,:})^TA_{i_k,:}}{||A_{i_k,:}||_2^2}\big)(Y^{(k)}-X_*)||_F^2+E_k\frac{||B_{i_k,:}^\perp-Z_{i_k,:}^{(k+1)}||_F^2}{||A_{i_k,:}||_2^2}\\  &=\sum_{i_k=1}^m\frac{||r_{i_k,:}||_2^2}{||r||_F^2}\big(trace\big((Y^{(k)}-X_*)^T(I-\frac{(A_{i_k,:})^TA_{i_k,:}}{||A_{i_k,:}||_2^2})(Y^{(k)}-X_*)\big)\big)\\
&+\sum_{i_k=1}^m\frac{||r_{i_k,:}||_2^2}{||r||_F^2}\frac{||B_{i_k,:}^\perp-Z_{i_k,:}^{(k+1)}||_F^2}{||A_{i_k,:}||_2^2}\\
    &=\sum_{i_k=1}^m\frac{||r_{i_k,:}||_2^2}{||r||_F^2}\big(trace\big((Y^{(k)}-X_*)^T(Y^{(k)}-X_*)\big)\big)\\
&+\sum_{i_k=1}^m\frac{||r_{i_k,:}||_2^2}{||r||_F^2}\frac{||B_{i_k,:}^\perp-Z_{i_k,:}^{(k+1)}||_F^2}{||A_{i_k,:}||_2^2}\\
    &-\sum_{i_k=1}^m\frac{||r_{i_k,:}||_2^2}{||r||_F^2}\big(trace\big((Y^{(k)}-X_*)^T\frac{(A_{i_k,:})^TA_{i_k,:}}{||A_{i_k,:}||_2^2}(Y^{(k)}-X_*)\big)\big)\\
    &=||Y^{(k)}-X_*||_F^2-\sum_{i_k=1}^m\frac{||r_{i_k,:}||_2^2}{||r||_F^2}\frac{||A_{i_k,:}(Y^{(k)}-X_*)||_F^2}{||A_{i_k,:}||_2^2}\\
&+\sum_{i_k=1}^m\frac{||r_{i_k,:}||_2^2}{||r||_F^2}\frac{||B_{i_k,:}^\perp-Z_{i_k,:}^{(k+1)}||_2^2}{||A_{i_k,:}||_2^2}\\
    &=||Y^{(k)}-X_*||_F^2-\sum_{i_k=1}^m\frac{||r_{i_k,:}||_2^2}{||r||_F^2}\frac{||\hat{B}_{i_k,:}-A_{i_k,:}Y^{(k)}||_2^2}{||A_{i_k,:}||_2^2}\\
&+\sum_{i_k=1}^m\frac{||r_{i_k,:}||_2^2}{||r||_F^2}\frac{||B_{i_k,:}^{\perp}-Z_{i_k,:}^{(k+1)}||_2^2}{||A_{i_k,:}||_2^2}\\
    &=||Y^{(k)}-X_*||_F^2+\sum_{i_k=1}^m\frac{||r_{i_k,:}||_2^2}{||r||_F^2}\frac{||B_{i_k,:}^\perp-Z_{i_k,:}^{(k+1)}||_2^2}{||A_{i_k,:}||_2^2}\\
&-\sum_{i_k=1}^m\frac{||r_{i_k,:}||_2^2}{||r||_F^2}\frac{||B_{i_k,:}^\perp+\hat{B}_{i_k,:}-A_{i_k,:}Y^{(k)}-Z_{i_k,:}^{(k+1)}+Z_{i_k,:}^{(k+1)}-B_{i_k,:}^\perp||_2^2}{||A_{i_k,:}||_2^2}.\\
\end{align*}
It follows from Lemma \ref{lemma2} that
\begin{align}\label{eq9}
     E_k||X^{(k+1)}-X_*||_F^2 \notag
     &\leq ||Y^{(k)}-X_*||_F^2+\sum_{i_k=1}^m\frac{||r_{i_k,:}||_2^2}{||r||_F^2}\frac{||B_{i_k,:}^\perp-Z_{i_k,:}^{(k+1)}||_2^2}{||A_{i_k,:}||_F^2} \\ \notag
     &-\sum_{i_k=1}^m\frac{||r_{i_k,:}||_2^2}{||r||_F^2}\frac{\alpha_1||B_{i_k,:}-A_{i_k,:}Y^{(k)}-Z_{i_k,:}^{(k+1)}||_2^2}{||A_{i_k,:}||_2^2}\\ \notag
     &+\sum_{i_k=1}^m\frac{||r_{i_k,:}||_2^2}{||r||_F^2}\frac{\beta_1||Z_{i_k,:}^{(k+1)}-B_{i_k,:}^\perp||_2^2}{||A_{i_k,:}||_2^2}\\ \notag
     &=||Y^{(k)}-X_*||_F^2-\alpha_1\frac{1}{||r||_F^2}\sum_{i_k=1}^m\frac{||r_{i_k,:}||_2^4}{||A_{i_k,:}||_2^2}\\ 
     &+(1+\beta_1)\sum_{i_k=1}^m\frac{||r_{i_k,:}||_2^2}{||r||_F^2}\frac{||B_{i_k,:}^\perp-Z_{i_k,:}^{(k+1)}||_2^2}{||A_{i_k,:}||_2^2}. 
\end{align}
We split the last formula of (\ref{eq9}) into three parts. Then, the second part can lead to the estimate
\begin{align*}
  \frac{1}{||r||_F^2}\sum_{i_k=1 }^m\frac{||r_{i_k,:}||_2^4}{||A_{i_k,:}||_2^2}
&\geq\frac{1}{||r||_F^2} \frac{(\sum_{i_k=1 }^m||r_{i_k,:}||_2^2)^2}{\sum_{i_k=1 }^m||A_{i_k,:}||_2^2}\\
&=\frac{1}{||r||_F^2}\frac{||r||_F^4}{||A||_F^2}=\frac{||r||_F^2}{||A||_F^2}\\
   & =\frac{||B-AY^{(k)}-Z^{(k+1)}||_F^2}{||A||_F^2}\\
&\geq \alpha_1\frac{||\hat{B}-AY^{(k)}||_F^2}{||A||_F^2}-\beta_1\frac{||B^\perp-Z^{(k+1)}||_F^2}{||A||_F^2}\\
    &=\alpha_1\frac{||A(Y^{(k)}-X_*)||_F^2}{||A||_F^2}-\beta_1\frac{||B^\perp-Z^{(k+1)}||_F^2}{||A||_F^2}\\
&\geq\alpha_1\frac{\lambda_{min}(A^TA)||Y^{(k)}-X_*||_F^2}{||A||_F^2}-\beta_1\frac{||B^\perp-Z^{(k+1)}||_F^2}{||A||_F^2}.
\end{align*}
where the first inequality follows from Lemma \ref {lemma1},  the second inequality is obtained by  Lemma \ref {lemma2}, and the following estimate
$||Az||_F^2\geq \lambda_{min}(A^TA)||z||_F^2$ is used. Since $(Y^{(k)}-X_*)_{:,j}\in R(A^T)$ (as proved in Section 1 $(i)$ of the supplementary file ),  the last inequality holds.

For the third part,  because of $(B^\perp-Z^{(k+1)})_{:,j}\in R(A)$, there  exists a vector $\hat{X}$ such that $B^\perp-Z^{(k+1)}=A\hat{X}$. Let $\hat{X}=A^+(B^\perp-Z^{(k+1)})$, then it results in
\begin{align*}
\sum_{i_k=1}^m\frac{||r_{i_k,:}||_2^2}{||r||_F^2}\frac{||B_{i_k,:}^\perp-Z_{i_k,:}^{(k+1)}||_2^2}{||A_{i_k,:}||_2^2}
    &= \sum_{i_k=1}^m\frac{||\hat{B}_{i_k,:}-A_{i_k,:}Y^{(k)}+B_{i_k,:}^\perp-Z_{i_k,:}^{(k+1)}||_2^2}{||r||_F^2}\frac{||B_{i_k,:}^\perp-Z_{i_k,:}^{(k+1)}||_2^2}{||A_{i_k,:}||_2^2}\\
    &=\sum_{i_k=1}^m\frac{||A_{i_k,:}(X_*-Y^{(k)}+\hat{X})||_2^2}{||r||_F^2}\frac{||B_{i_k,:}^\perp-Z_{i_k,:}^{(k+1)}||_2^2}{||A_{i_k,:}||_2^2}\\
    &\leq \sum_{i_k=1}^m\frac{||A_{i_k,:}||_2^2||X_*-Y^{(k)}+\hat{X}||_F^2}{||r||_F^2}\frac{||B_{i_k,:}^\perp-Z_{i_k,:}^{(k+1)}||_2^2}{||A_{i_k,:}||_2^2}\\
    &=\sum_{i_k=1}^m\frac{||X_*-Y^{(k)}+\hat{X}||_F^2}{||A(X_*-Y^{(k)}+\hat{X})||_F^2}||B_{i_k,:}^\perp-Z_{i_k,:}^{(k+1)}||_2^2\\
    &\leq \frac{||X_*-Y^{(k)}+\hat{X}||_F^2}{\lambda_{min}(A^TA)||X_*-Y^{(k)}+\hat{X}||_F^2}\sum_{i_k=1}^m ||B_{i_k,:}^\perp-Z_{i_k,:}^{(k+1)}||_2^2\\
    &\leq \frac{||B^\perp-Z^{(k+1)}||_F^2}{\lambda_{min}(A^TA)}
\end{align*}
 Here, since $(X_*-Y^{(k)}+\hat{X})_{:,j}\in R(A^T)$ (as proved in Section 1 $(ii)$ of the supplementary file ), the second inequality $||A(X_*-Y^{(k)}+\hat{X})||_F^2\geq\lambda_{min}(A^TA)||X_*-Y^{(k)}+\hat{X}||_F^2$ holds. Thus we have
\begin{align}\label{eq10}
     E_k||X^{(k+1)}-X_*||_F^2 \notag
     &\leq 
    ||Y^{(k)}-X_*||_F^2-\alpha_1(\alpha_1\frac{\lambda_{min}(A^TA)||Y^{(k)}-X_*||_F^2}{||A||_F^2}-\beta_1\frac{||B^\perp-Z^{(k+1)}||_F^2}{||A||_F^2})\\ \notag
    &+(1+\beta_1)\frac{1}{\lambda_{min}(A^TA)}||B^\perp-Z^{(k+1)}||_F^2\\ \notag
    &=(1-\frac{\alpha_1^2\lambda_{min}(A^TA)}{||A||_F^2})||Y^{(k)}-X_*||_F^2+(\frac{\alpha_1\beta_1}{||A||_F^2}+\frac{1+\beta_1}{\lambda_{min}(A^TA)})||B^\perp-Z^{(k+1)}||_F^2\\ \notag
    &=\delta||X^{(k)}+\gamma(X^{(k)}-X^{(k-1)})-X_*||_F^2+\theta||B^\perp-Z^{(k+1)}||_F^2\\ \notag
    &=\delta||(1+\gamma)X^{(k)}-\gamma X^{(k-1)}-(1+\gamma-\gamma)X_*||_F^2+\theta||B^\perp-Z^{(k+1)}||_F^2\\ \notag
    &=\delta||(1+\gamma)(X^{(k)}-X_*)-\gamma(X^{(k-1)}-X_*)||_F^2+\theta||B^\perp-Z^{(k+1)}||_F^2\\ \notag
    &=\delta(1+\gamma)^2||X^{(k)}-X_*||_F^2+\gamma^2||X^{(k-1)}-X_*||_F^2\\ \notag
&-2 \delta\gamma(1+\gamma)trace[(X^{(k)}-X_*)^T(X^{(k-1)}-X_*)]
    +\theta||B^\perp-Z^{(k+1)}||_F^2\\ \notag
     &\leq \delta(1+\gamma)^2||X^{(k)}-X_*||_F^2+\gamma^2||X^{(k-1)}-X_*||_F^2\\ \notag
    &+\gamma(1+\gamma)(||X^{(k)}-X_*||_F^2+||X^{(k-1)}-X_*||_F^2)
    +\theta||B^\perp-Z^{(k+1)}||_F^2\\ \notag
    &=\delta[(1+\gamma)^2+\gamma(1+\gamma)]||X^{(k)}-X_*||_F^2+\theta||B^\perp-Z^{(k+1)}||_F^2\\ \notag
    &+ \delta(\gamma^2+\gamma(\gamma+1))||X^{(k-1)}-X_*||_F^2\\ \notag
    &=\delta(2\gamma^2+3\gamma+1)||X^{(k)}-X_*||_F^2\\ \notag
    &+\delta\gamma(2\gamma+1)|X^{(k-1)}-X_*||_F^2+\theta||B^\perp-Z^{(k+1)}||_F^2\\ 
    &=\eta_1||X^{(k)}-X_*||_F^2+\eta_2||X^{(k-1)}-X_*||_F^2+\theta||B^\perp-Z^{(k+1)}||_F^2.
\end{align}
where the second inequality uses the parallelogram identity $2\langle C,D\rangle_F = \|C\|_F^2 - \|C-D\|_F^2+ \|D\|^2_F$ and the inequality
$\|C -D\|_F^2 \leq \left(\|C\|_F + \|D\|_F\right)^2 \leq 2\|C\|_F^2 + 2\|D\|_F^2$.

From the expression of $\delta$, it is easy to obtain that $0<\delta<1$.  By taking full expectation on both sides of the above inequality (\ref{eq10}), we
can get
\begin{equation*}
    E||X^{(k+1)}-X_*||_F^2\leq \eta_1E||X^{(k)}-X_*||_F^2+\eta_2E||X^{(k-1)}-X_*||_F^2+\theta E||B^\perp-Z^{(k+1)}||_F^2.
\end{equation*}
Thus, by Lemma \ref{thm1}, we have
\begin{align*}
    E||X^{(k)}-X_*||_F^2
    &\leq \eta_1E||X^{(k-1)}-X_*||_F^2+\eta_2E||X^{(k-2)}-X_*||_F^2+\theta E||B^\perp-Z^{(k)}||_F^2\\
    &\leq \eta_1E||X^{(k-1)}-X_*||_F^2+\eta_2E||X^{(k-2)}-X_*||_F^2+\theta w^k\lambda_{max}(A^TA)||X_*||_F^2\\
    &=\eta_1E||X^{(k-1)}-X_*||_F^2+\eta_2E||X^{(k-2)}-X_*||_F^2+a_3.
\end{align*}
Let $a_3=\theta\big(1-\frac{\lambda_{min}(A^TA)}{||A||_F^2}\big)^k\lambda_{max}(A^TA)||X_*||_F^2>0$. Since the momentum parameter satisfies $0\leq\gamma<\frac{1-\sqrt{\delta}}{2\sqrt{\delta}}$ (as proved in Section 2 of the supplementary file), we have $\eta_1 \geq 0 $, $ \eta_2 \geq 0 $, $0<\eta_1+\eta_2<1$ , and at least one of the coefficients $ \eta_1 $ and $\eta_2 $ is positive. Meanwhile, Algorithm \ref{algorithm7} is equivalent to the iteration scheme with an incremented index, provided that the initial values satisfy $X^{(1)}=X^{(0)}$. Thus by Lemma \ref{lemma3}, the conclusion follows
\begin{equation*} 
    E||X^{(k)}-X_*||_F^2\leq q^{(k)}(1+\xi)E||X^{(0)}-X_*||_F^2+\frac{a_3}{1-q},
\end{equation*}
where $0<q=\frac{\eta_1+\sqrt{\eta_1^2+4\eta_2}}{2}<1$ and $\xi=q-\eta_1\geq0$.
\end{proof}
\begin{corollary}
	Assume that the parameters $\alpha_1$ and $\beta_1$ are specified in Lemma \ref{lemma2}. Denote $\delta=1-\frac{\alpha_1^2\lambda_{min}(A^TA)}{||A||_F^2},\theta=\frac{\alpha_1\beta_1}{||A||_F^2}+\frac{1+\beta_1}{\lambda_{min}(A^TA)}$  If the initial matrices $X^{(0)}$ satisfies $X^{(0)}_{:,j}\in R(A^T)(j=1,2,\cdots, p)$ and $Z^{(0)}=B$, the iteration sequence $\{X^{(k)}\}_{k=0}^\infty$ generated by Algorithm \ref{algorithm6} converges to the least-squares solution $X_*=A^+B$. Moreover, the solution error in expectation satisfies:
	\begin{equation*}
		E||X^{(k)}-X_*||_F^2\leq \delta^{(k)}||X^{(0)}-X_*||_F^2+\frac{a_3}{1-\delta},
	\end{equation*}
	where 
	 $a_3=\theta \big(1-\frac{\lambda_{min}(A^TA)}{||A||_F^2}\big)^k\lambda_{max}(A^TA)||X_*||_F^2$.
\end{corollary}
\begin{proof}

Setting $\gamma=0$ in Theorem \ref{thm2}, we obtain the conclusion.
\end{proof}
	
\section{Numerical experiments}\label{sec3}
In this section, we test the performance of these four methods: REKIAX,  REGSIAX, DREK and  MDREK. All our experiments are completed using the software Matlab (Version R2024a). The computer used is equipped with a Windows 11 operating system, 16G memory, and a 2.60GHz central processing unit (13th Gen Intel(R) Core(TM) i7-13650HX).

 For the inconsistent system \eqref{eq2}, we set \( B = AX_* + R_1 \) and $X_* = \text{pinv}(A)* B$, where $R_1=10^{-5}\times randn(m,p)$. The number of IT and the CPU are used as the main data for comparison with the tested methods. All experiments start from $X_0=0$ and $Z_0=B$, and the experiments stop iterating when the relative solution error (RSE) satisfies 
\begin{equation*}
	RES=\frac{||X_k-X_*||_F^2}{||X_*||_F^2}\leq 10^{-6}
\end{equation*}
or the iteration steps exceed $5\times 10^4$. The real-world sparse data come from the Florida sparse matrix collection\cite{Collection}. Table \ref{tab4}
lists the features of these sparse matrices. The density is defined as
\begin{equation*}
	density=\frac{the\ number\ of \ non-zero\ elements\ of\ a\ m-by-n\ matrix}{mn},
\end{equation*}
which indicates the sparsity of the corresponding matrix.
To ensure the reliability of the experimental results,
the number of iterations and  computation time in the numerical results represent the average values of 10 repeated runs.
%%合成数据实验
\begin{example}\label{em1} Random matrix. We consider a series of matrices $A$ generated via MATLAB's built-in functions in two ways: (1) $A = \text{randn}(m,n) $; (2) rank-deficient matrices $ A$ constructed as $ A = \text{randn}(m/2,n) $ followed by $ A = [A, A] $. For both cases, we set $ X = \text{randn}(n,p) $, $ B = AX + R$, and $R = \mu \times \text{randn}(m,p) $ with $\mu = 1 \times 10^{-5}$. For the momentum parameter $\beta $, we set its search range to $ (0,1) $ with a step size of $0.02 $. We iteratively test different values of $\beta$, record the final residual of the MDREK method, and select the value of $\beta$ corresponding to the minimum final residual as the momentum parameter for the corresponding matrix dimension.
	
	The numerical results of all methods are reported in Tables \ref{tab2}, \ref{tab3} and Figs \ref{fig1}$-$\ref{fig3}. From them, it is easy to see that whether the coefficient matrix A is full rank or rank-deficient, both the proposed DREK  and MDREK methods converge much faster than the other  two methods. In particular, the MDREK method requires the least number of IT and CPU time.
\end{example}
%biao2
\begin{table}
	\centering
	\caption{The average of IT and CPU of the four methods for Example \ref{em1} with full rank $A$}
	\label{tab2}
	\begin{tabular}{ccccccccc}
		\hline
		m&n&p& rank(A)& &REGSIAX &REKIAX&DREK&MDREK\\
		\hline
		30&50&30&30&IT&3708&3765&1747&$\mathbf{1509}$ ($\beta=0.75$) \\
		& & & &CPU&0.0738&0.0767&0.0475&$\mathbf{0.0404}$($\beta=0.75$) \\  
		50&30&30&30&IT&3960&3934&1771&$\mathbf{1502}$($\beta=0.51$) \\
		& & & &CPU&0.0746&0.0753&0.0451&$\mathbf{0.0395}$($\beta=0.51$) \\ 
		
		60&80&60&60&IT&16974&16847&8408&$\mathbf{7110}$($\beta=0.59$)\\ & & & &CPU&0.6247&0.6319&0.5721&$\mathbf{0.5197}$($\beta=0.59$)\\
		80&60&60&60&IT&25507&25166&10811&$\mathbf{8980}$($\beta=0.67$)\\ & & & &CPU&0.8258&0.9081&0.7516&$\mathbf{0.6338}$($\beta=0.67$)\\
		80&100&80&80&IT&42411&42641&19711&$\mathbf{17443}$($\beta=0.25$)\\
		& & & &CPU&1.9544&1.9774&1.8544&$\mathbf{1.7746}$($\beta=0.25$)\\
		100&80&80&80&IT&47503&47830&22265&$\mathbf{19857}$($\beta=0.27$)\\
		& & & &CPU&2.0212&2.1262&2.0039&$\mathbf{1.9510}$($\beta=0.27$)\\
		\hline
	\end{tabular}
\end{table}
%图1
\begin{figure}[H]   %[!t]:位置参数，指定图形应该放置在页面的顶部（top）
	\centering
	\includegraphics[width=5in]{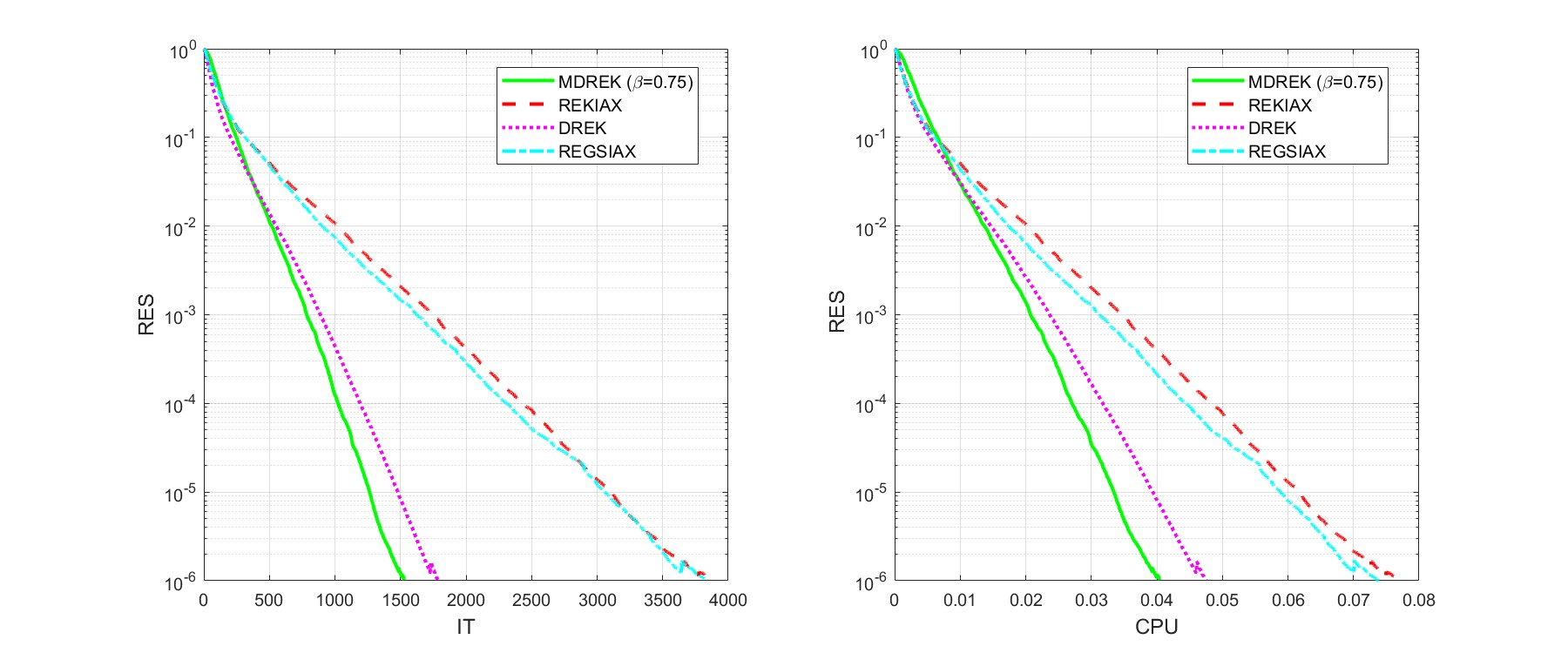}
	\caption[d]{The left is RES vs IT, and the right is RES vs CPU(s) for four different methods in Example \ref{em1} with m=30, n=50, p=30, rank(A)=30,  $\beta=0.75$
	}\label{fig1}
\end{figure}
%图2
\begin{figure}[!h]   %[!t]:位置参数，指定图形应该放置在页面的顶部（top）
	\centering
	\includegraphics[width=5in]{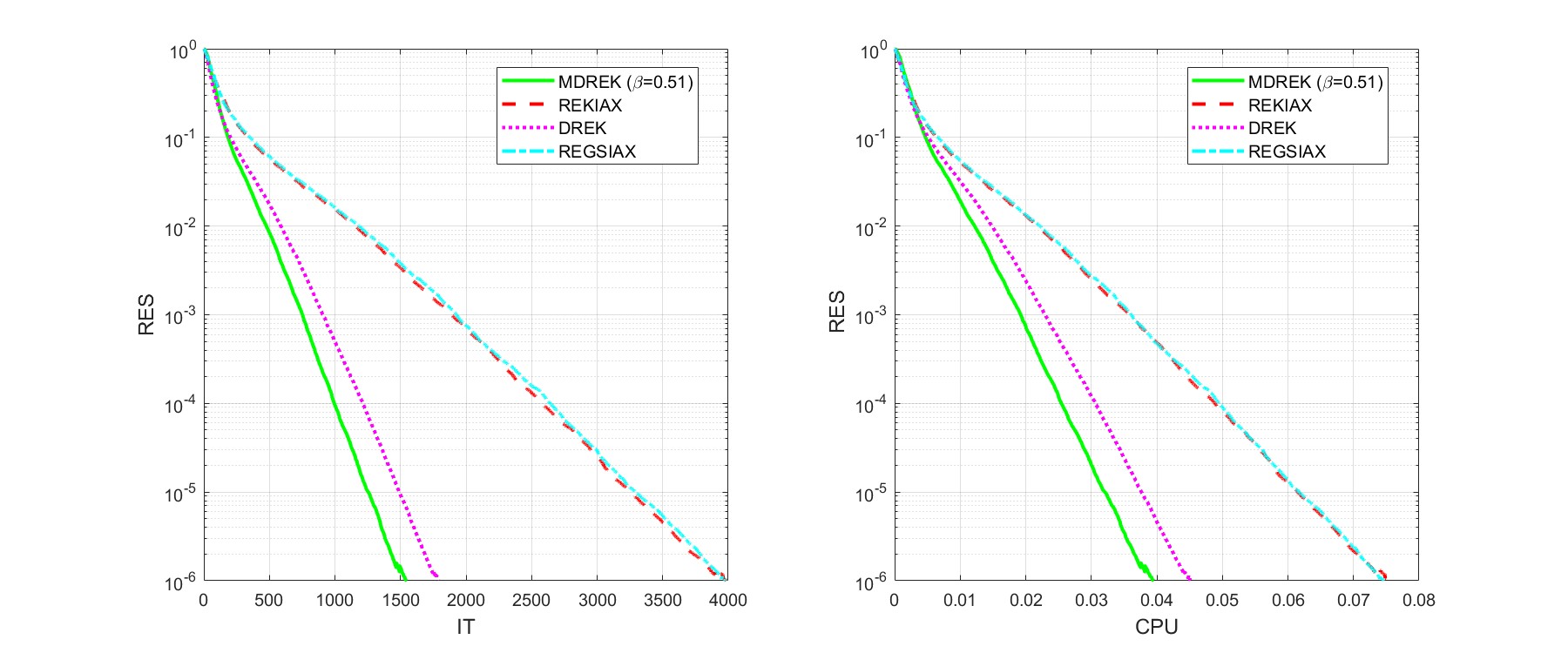}
	\caption[d]{The left is RES vs IT, and the right is RES vs CPU(s) for four different methods in Example \ref{em1} with  m=50, n=30, p=30, rank(A)=30, $\beta=0.51$
	}\label{fig2}
\end{figure}
%图3
\begin{figure}[!h]   %[!t]:位置参数，指定图形应该放置在页面的顶部（top）
	\centering
	\includegraphics[width=5in]{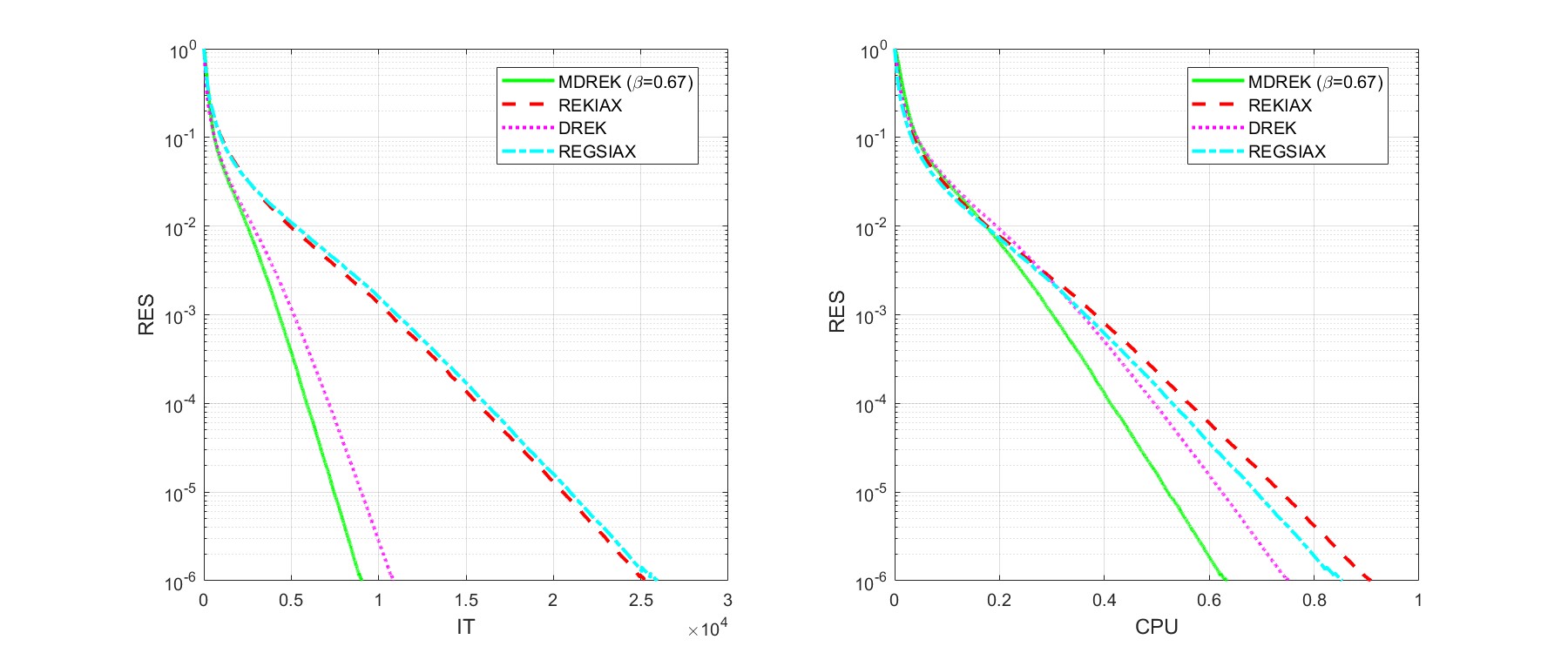}
	\caption[d]{The left is RES vs IT, and the right is RES vs CPU(s) for four different methods in Example \ref{em1} with  $m=80,n=60,p=60,rank(A)=60,  \beta=0.67$}\label{fig3}
\end{figure}

%biao2
\begin{table}
	\centering
		\caption{The average of IT and CPU of the four methods for Example \ref{em1} with deficient rank $A$}
	\label{tab3}
	\begin{tabular}{ccccccccc}
		\hline
		m&n&p& rank(A)& &REGSIAX &REKIAX&DREK&MDREK\\
		\hline
30&50&30&25&IT&17525&17012&7731&$\mathbf{6350}$($\beta=0.85$)\\
		& & & &CPU&0.3935&0.3973&0.2915&$\mathbf{0.2725}$ ($\beta=0.85$)\\
		50&30&30&25&IT&5183&5211&2738&$\mathbf{2334}$($\beta=0.47$)\\
		& & & &CPU&0.0967&0.1000&0.0695&$\mathbf{0.0616}$ ($\beta=0.47$)\\
		60&80&60&40&IT&10677&10630&4347&$\mathbf{3622}$ ($\beta=0.87$)\\
		& & & &CPU&0.3935&0.3973&0.2915&$\mathbf{0.2725}$ ($\beta=0.87$)\\
		80&60&60&40&IT&9756&10037&4374&$\mathbf{3634}$ ($\beta=0.73$)\\
		& & & &CPU&0.3231&0.3598&0.2963&$\mathbf{0.2622}$ ($\beta=0.73$)\\
		
		80&100&80&50&IT&10108&10023&4141&$\mathbf{3708}$ ($\beta=0.25$)\\
		& & & &CPU&0.4703&0.4568&0.3971&$\mathbf{0.3778}$ ($\beta=0.25$)\\
		100&80&80&50&IT&7407&7119&3368&$\mathbf{2756}$ ($\beta=0.53$)\\
		& & & &CPU&0.3204&0.3161&0.3105&$\mathbf{0.2905}$($\beta=0.53$)\\
		\hline
	\end{tabular}
\end{table}
%图4
\begin{figure}[!h]   %[!t]:位置参数，指定图形应该放置在页面的顶部（top）
	\centering
	\includegraphics[width=5in]{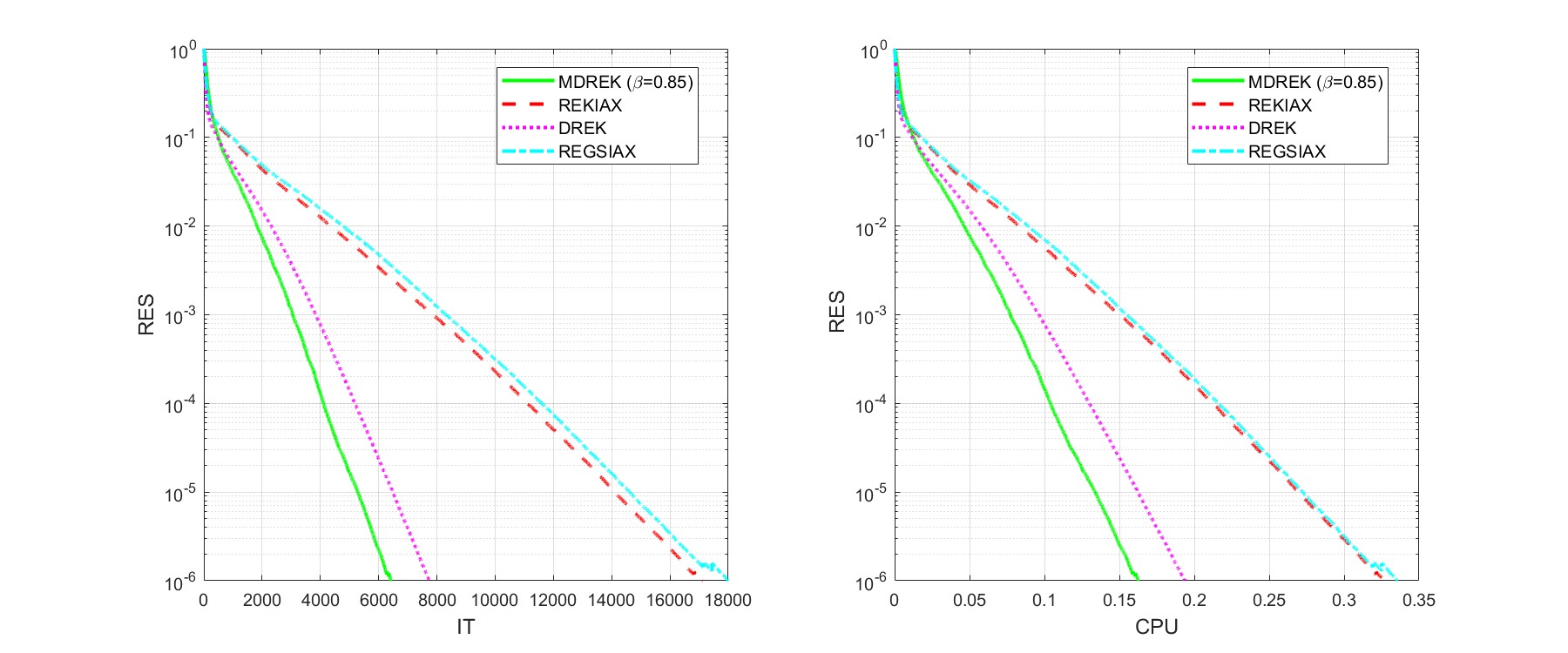}
	\caption[d]{The left is RES vs IT, and the right is RES vs CPU(s) for four different methods in Example \ref{em2} with  $m=30,n=50,p=30,rank(A)=25, \beta=0.85$}\label{fig4}
\end{figure}
%图5
\begin{figure}[!h]   %[!t]:位置参数，指定图形应该放置在页面的顶部（top）
	\centering
	\includegraphics[width=5in]{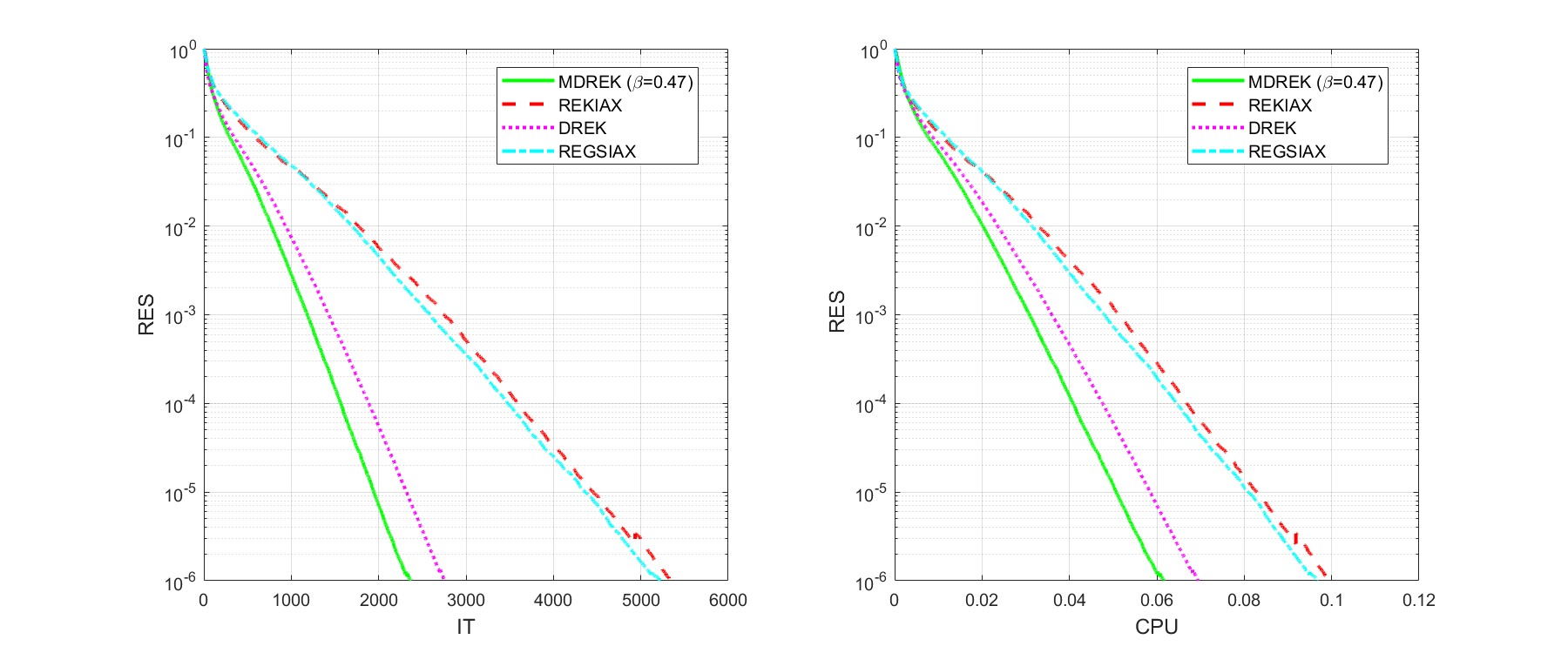}
	\caption[d]{The left is RES vs IT, and the right is RES vs CPU(s) for four different methods in Example \ref{em2} with  $m=50,\ n=30,\ p=30,\ rank(A)=25,\ \beta=0.47$}\label{fig5}
\end{figure}
%图6
\begin{figure}[!h]   %[!t]:位置参数，指定图形应该放置在页面的顶部（top）
	\centering
	\includegraphics[width=5in]{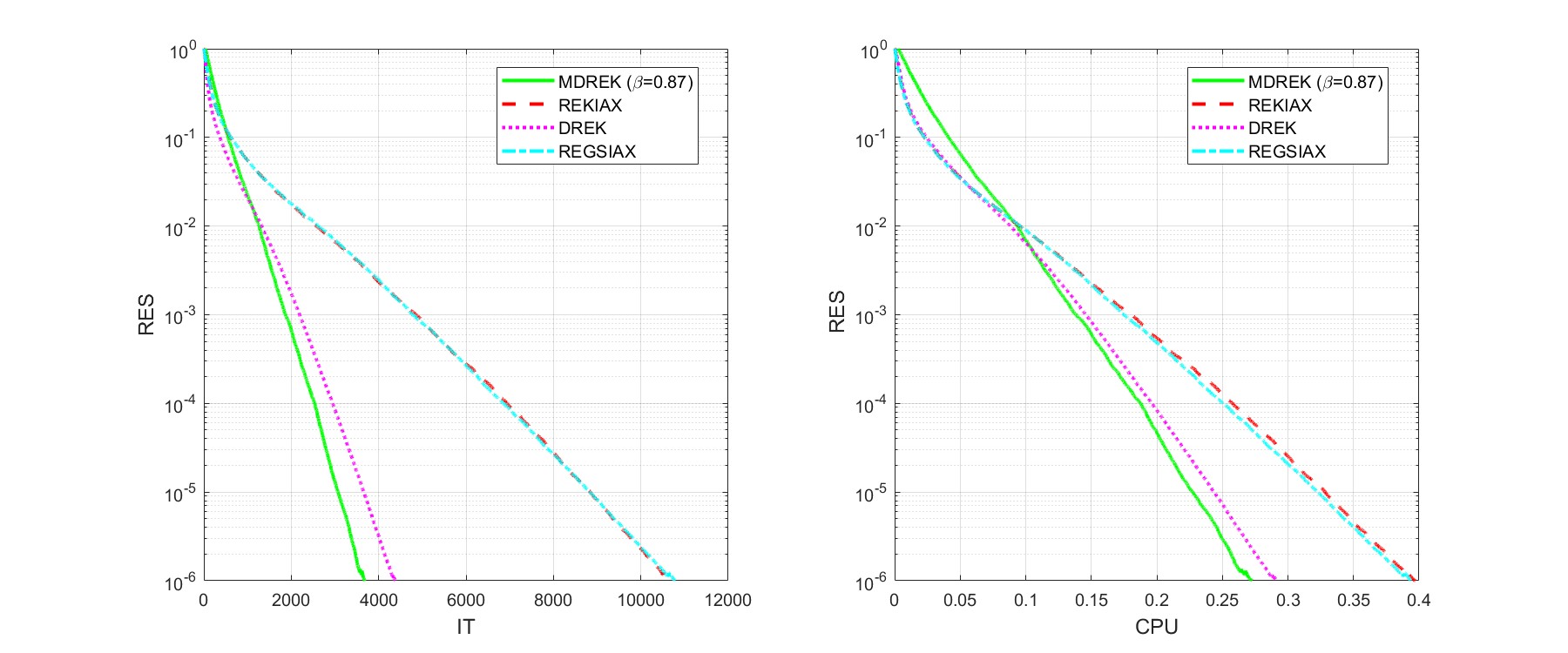}
	\caption[d]{The left is RES vs IT, and the right is RES vs CPU(s) for four different methods in Example \ref{em2} with  $m=60,n=80,p=60,rank(A)=40,  \beta=0.87$}\label{fig6}
\end{figure}

\begin{example}\label{em2} Real-world matrix.
	We compare these four different algorithms using various types of real matrix data. The relevant information about the real matrices is presented in Table~\ref{tab4}. Let $X = \text{randn}(n,p)$ with $p=10$ and $B = AX + R$ where $p=10$, $R = \mu \times \text{randn}(m,p)$, $\mu = 1 \times 10^{-5}$. The momentum parameter is set to $\beta=0.25$. As observed from Table~\ref{tab5} and Figures~\ref{fig4}--\ref{fig6}, the same conclusion as in the previous example is reached: both the proposed DREK and MDREK methods converge significantly faster than the other two methods. In particular, the MDREK method requires the smallest number of IT and the least CPU time.
		
\end{example}
%真实矩阵信息（表3）
\begin{table}[]
	\centering
	\caption{Properties of real matrices in Example~\ref{em2}}\label{tab4}
	\begin{tabular}{ccccc}
		\hline
		Name & $m\times n$ & Density & Rank & $\kappa(A)$\\
		\hline
		can$\_$144 & $144\times144$ & 6.3$\%$ & 96 & 1.01\\
		lp$\_$sc50b & $78\times50$ & 3.8$\%$ & 50 & 1.00\\
		divorce & $50\times9$ & 50$\%$ & 9  & 19.39\\
		\hline
	\end{tabular}
	\label{tab:placeholder}
\end{table}
%biao5
\begin{table}
	\centering
	\caption{The average IT and CPU of four methods for Example \ref{em2} }
	\label{tab5}
	\begin{tabular}{ccccccc}
		\hline
		real matrix& &REGSIAX&REKIAX&DREK&MDREK($\beta=0.25$)\\
		\hline
		can$\_$144& IT&50000&50000&22898&$\mathbf{20589}$ \\
		&CPU&0.9776&0.9773&0.7479&$\mathbf{0.6957}$\\
		lp$\_$sc50b& IT &17740&17820&3906&$\mathbf{3350}$ \\
		&CPU&0.3141&0.3169&0.0918&$\mathbf{0.0874}$\\
		lp$\_$sc50b$^T$& IT & 18565&18061&3691&$\mathbf{3320}$ \\
		&CPU&0.3184&0.2970&0.0855&$\mathbf{0.0787}$\\
		divorce& IT &4583&4141&1071&$\mathbf{959}$ \\
		&CPU&0.0723&0.0653&0.0203&$\mathbf{0.0193}$\\
		divorce$^T$& IT &4473&3951&998&$\mathbf{895}$ \\
		&CPU&0.0803&0.0755&0.0259&$\mathbf{0.0256}$\\
		\hline
	\end{tabular}
\end{table}
\section{Conclusion}\label{sec4}
In this paper, we propose the DREK method and its Nesterov momentum variant (MDREK) for solving the linear matrix equation system $AX=B$. Under mild conditions, we prove the convergence of the proposed methods and derive the upper bound for their convergence rates. Numerical experiments show that the CPU time of the new methods is less than that of the REKIAX and REGSIAX methods, which demonstrates the efficiency of the new methods. In future work, we will investigate their block versions.

\noindent {\bf Author Contributions}
Wendi Bao: Conceptualization, Methodology.
Jing Li: Writing – original draft preparation, Numerical experiments.
Lili Xing: Visualization, Numerical experiments. 
Weiguo Li: Investigation.
Jichao Wang: Writing – Reviewing and Editing.

\noindent {\bf Funding}
This work was supported by the National Natural Science Foundation of China (grant numbers 61931025, 42176011), and the Fundamental Research Funds for the Central Universities of China (grant number 24CX03001A).

\noindent {\bf Competing Interests}
We declare that we have no known competing financial interests or personal relationships that could have appeared to influence the work reported in this paper.

\backmatter
\bibliography{sn-bibliography}
\end{document}